\documentclass[11pt, intlimits]{amsart}
\usepackage{amsmath,amssymb}
\usepackage{graphicx}
\usepackage{verbatim}
\usepackage{amsmath,amssymb}
\usepackage{graphicx}
\usepackage{amsbsy}
\usepackage{amsfonts}
\usepackage{amsmath}
\usepackage{amsthm}
\usepackage{amssymb}
\usepackage{enumitem}
\usepackage{ytableau}

\usepackage{hyperref}
\usepackage{color}
\usepackage{graphicx}

\begin{document}

\newtheorem{theorem}{Theorem}[section]
\newtheorem{lemma}[theorem]{Lemma}
\newtheorem{proposition}[theorem]{Proposition}
\newtheorem{cor}[theorem]{Corollary}
\newtheorem{defn}[theorem]{Definition}
\newtheorem*{remark}{Remark}
\newtheorem{conj}[theorem]{Conjecture}

\numberwithin{equation}{section}

\newcommand{\Z}{{\mathbb Z}} %cph changed from \mathbf
\newcommand{\Q}{{\mathbb Q}}
\newcommand{\R}{{\mathbb R}}
\newcommand{\C}{{\mathbb C}}
\newcommand{\N}{{\mathbb N}}
\newcommand{\FF}{{\mathbb F}}
\newcommand{\fq}{\mathbb{F}_q}
\newcommand{\rmk}[1]{\footnote{{\bf Comment:} #1}}

\renewcommand{\mod}{\;\operatorname{mod}}
\newcommand{\ord}{\operatorname{ord}}
\newcommand{\TT}{\mathbb{T}}
\renewcommand{\i}{{\mathrm{i}}}
\renewcommand{\d}{{\mathrm{d}}}
\renewcommand{\^}{\widehat}
\newcommand{\HH}{\mathbb H}
\newcommand{\Vol}{\operatorname{vol}}
\newcommand{\area}{\operatorname{area}}
\newcommand{\tr}{\operatorname{tr}}
\newcommand{\norm}{\mathcal N} % norm =(\frac{ n+\sqrt{n^2-4}} 2)^2
\newcommand{\intinf}{\int_{-\infty}^\infty}
\newcommand{\ave}[1]{\left\langle#1\right\rangle} %  average
\newcommand{\Var}{\operatorname{Var}}
\newcommand{\Prob}{\operatorname{Prob}}
\newcommand{\sym}{\operatorname{Sym}}
\newcommand{\disc}{\operatorname{disc}}
\newcommand{\CA}{{\mathcal C}_A}
\newcommand{\cond}{\operatorname{cond}} % conductor
\newcommand{\lcm}{\operatorname{lcm}}
\newcommand{\Kl}{\operatorname{Kl}} %Kloosterman sum
\newcommand{\leg}[2]{\left( \frac{#1}{#2} \right)}  % Legendre symbol
\newcommand{\Li}{\operatorname{Li}}

\newcommand{\ee}{ \operatorname{e}}
\newcommand{\sumstar}{\sideset \and^{*} \to \sum}
\newcommand{\ooo}{\operatorname{\ddot o}}
\newcommand{\uuu}{\operatorname{\ddot u}}
\newcommand{\uuuu}{\operatorname{\ddot U}}
\newcommand{\LL}{\mathcal L} %L-function of u
\newcommand{\sumf}{\sum^\flat}
\newcommand{\Hgev}{\mathcal H_{2g+2,q}}
\newcommand{\USp}{\operatorname{USp}}
\newcommand{\conv}{*}
\newcommand{\dist} {\operatorname{dist}}
\newcommand{\CF}{c_0} % Fejer constant
\newcommand{\kerp}{\mathcal K}

\newcommand{\Cov}{\operatorname{cov}}

\newcommand{\ES}{\mathcal S} % sums over AP
\newcommand{\EN}{\mathcal N} % sum over short intervals
\newcommand{\EM}{\mathcal M} % sum over pols of deg n
\newcommand{\Sc}{\operatorname{Sc}} %Secular coefficients
\newcommand{\Ht}{\operatorname{Ht}}

\newcommand{\E}{\operatorname{E}} % expectation
\newcommand{\sign}{\operatorname{sign}} %sign

\newcommand{\divid}{d} % the divisor function
\newcommand{\inv}{\theta}

\newcommand{\vleq}{\rotatebox[origin=c]{90}{$\leq$}}
\newcommand{\odd}{\operatorname{odd}}
\newcommand{\even}{\operatorname{even}}

\theoremstyle{plain}

\newcommand\MOD{\textrm{ (mod }}

\newcommand\spann{\mathrm{span}}
\newcommand\primeE{\mathop{\vcenter{\hbox{\relsize{+2}$\mathbf{E}$}}}}
\newcommand\Ss{\mathcal{S}}
\newcommand\Dd{\mathcal{D}}
\newcommand\Uu{\mathcal{U}}
\newcommand\supp{\mathrm{supp}\;}
\newcommand\sgn{\mathrm{sgn}}
\newcommand\bb{\mathbb}
\newcommand\pvint{-\!\!\!\!\!\!\!\int_{-\infty}^\infty}
\def\res{\mathop{\mathrm{Res}}}
\newcommand\primesum{\sideset{}{'}\sum}
\newcommand\primeprod{\sideset{}{'}\prod}

\newtheorem{ex}[theorem]{Example}

\newcommand{\Poly}{{\operatorname{Poly}}}
\renewcommand{\mod}{\;\operatorname{mod}}
\renewcommand{\i}{{\mathrm{i}}}
\renewcommand{\d}{{\mathrm{d}}}
\renewcommand{\^}{\widehat}

\newcommand{\diag}{\operatorname{diag}}

\title[Square-full polynomials]{Square-full polynomials in short intervals and in arithmetic progressions }

\author{ E. Roditty-Gershon }

\address{School of Mathematics, University of Bristol, Bristol BS8 1TW, UK}
\email{er14265@bristol.ac.uk}

\date{\today}

%\thanks{ to be added }

\begin{abstract}
We study the variance of sums of the indicator function of square-full polynomials in both arithmetic progressions and short intervals. Our work is in the context of the ring $\fq[T]$ of polynomials over a finite field $\fq$ of $q$ elements, in the limit $q\rightarrow\infty$. We use a recent equidistribution result due to N. Katz to express these variances in terms of triple matrix integrals over the unitary group, and evaluate them.
\end{abstract}

\maketitle

\tableofcontents

\section{Introduction}
A positive integer $n$ is called a square-full number if $p^{2}|n$ for every prime factor $p$ of $n$. Denote by $\alpha_{2}$ the indicator function of square-full numbers, i.e.
\begin{equation}\label{indicator}
\alpha_{2}(n) = \begin{cases} 1 & \mbox{if }n \mbox{ is square-full }
                           \\ 0 & \mbox{otherwise} \end{cases}
\end{equation}
Let $\mathcal{A}(x)$ be the number of square full integers not exceeding $x$. In 1935, Erd$\ooo$s and Szekeres \cite{Erdos} proved
\begin{equation}
\mathcal{A}(x)=\frac{\zeta(3/2)}{\zeta(3)}x^{1/2}+O_{\epsilon}(x^{1/3+\epsilon}).
\end{equation}
Bateman and Grosswald \cite{{Bateman}} improved this result in 1958. They obtained
\begin{equation} \label{Bateman and Grosswald}
\mathcal{A}(x)=\frac{\zeta(3/2)}{\zeta(3)}x^{1/2}+\frac{\zeta(2/3)}{\zeta(2)}x^{1/3}+O(x^{1/6}\cdot \mathrm{e^{-c(\log^{3}x/\log\log x)^{\frac{1}{5}}}})
\end{equation}
where $c$ is a positive absolute constant. They also made the observation that any improvment of the exponent $\frac{1}{6}$ would imply that $\zeta(s)\neq 0$ for $\Re s>1-\delta ~~~(\delta>0).$ There is a very long history of studies and conditional improvements of the error term in the above formula (see \cite{Cai},\cite{Cao1},\cite{Cao2},\cite{Liu1},\cite{Suryanarayana},\cite{Wu1},\cite{Wu2},\cite{Zhu-Yu}).

It follows from \eqref{Bateman and Grosswald} that for intervals of length
$H>x^{2/3 +\epsilon}$ we have
\begin{equation}
\sum_{x\leq n \leq x+H}\alpha_{2}(n)\sim \frac{\zeta(3/2)}{2\zeta(3)}\cdot H/x^{1/2}.
\end{equation}
Various authors used exponential sum techniques to reduce the lower bound on $H$ for which this asymptotic is valid. Heath-Brown \cite{HEATH-BROWN} proved it with for $H>x^{\eta+\epsilon}$ with $\eta=0.6318\ldots$. Liu \cite{Liu2} proved it with $\eta=0.6308\ldots$.
Filaseta and Trifonov\cite{Filaseta} found a simpler approach, using real instead of complex analysis, and obtained the exponent $\eta=0.6282\ldots.$ In \cite{Huxley}, Huxley and Trifonov improve this to $H\geq \frac{1}{\epsilon}x^{5/8}(\log x)^{5/16}$.

Concerning the distribution of square full numbers in arithmetic progressions, the most recent result is due to Munsch \cite{Munsch}. By evaluating character sums, he showed that
\begin{equation}
\sum_{\substack{n\leq x\\n=a \mod q}}\alpha_{2}(n)\sim\frac{\zeta(3/2)}{\zeta(3)}\frac{A_{a,q}}{q}x^{1/2}+
\frac{\zeta(2/3)}{\zeta(2)}\frac{B_{a,q}}{q}x^{1/3}
\end{equation}
where
$$A_{a,q}=\prod_{p|q}(1-\frac{1}{p^{3}})^{-1}\sum_{\chi \in X_{2}}\chi(a) \frac{L(3/2,\chi)}{\zeta(3/2)}$$
and
$$B_{a,q}=\prod_{p|q}(1-\frac{1}{p^{2}})^{-1}\sum_{\chi \in X_{3}}\bar{\chi}(a) \frac{L(2/3,\chi)}{\zeta(2/3)}$$
with $X_{2}$ and $X_{3}$ being the set of all quadratic and cubic characters mod $q$ respectively and $L(s,\chi)$ is the L-function attached $\chi$.

The goal of this paper is to study the fluctuations of the analogous sums in the function field settings. Namely, we study the variance of the sum of $\alpha_{2}$ in arithmetic progressions or in short intervals, in the context of the ring $\fq[T]$ of polynomials over a finite field $\fq$ of $q$ elements, in the limit $q\rightarrow\infty$. In our setting we succeed in giving definitive answers in both cases.

Our approach involves converting the problem to one about the correlation of zeros of a certain family of L-functions, and then using an equidistribution result of Katz which holds in the limit $q \rightarrow \infty$.

%%%%%%%%%%%%%%%%%%%%%%%%%%%%%%%%%%%%%%%%%%%%%%%%%%%%%%%%%%%%%%%%%%%%%%%%%%%%%%%%%%%%%%%%%%%%%%%%%%%
\section{Square-full Polynomials}
Let $\fq$ be a finite field of an odd cardinality $q$, and let $M_{n}$ be the set of all monic polynomials of degree $n$ with coefficients in $\fq$. In analogy to numbers, we say that $f\in \EM_n$ is a square-full polynomial if for every polynomial $P\in \EM_n$ that divides $f$, $P^{2}$ also divides $f$. We denote by $\alpha_{2}$ the indicator function of square-full polynomials, i.e.
\begin{equation}\label{indicator}
\alpha_{2}(f) = \begin{cases} 1 & \mbox{if }f \mbox{ is square-full }
                           \\ 0 & \mbox{otherwise} \end{cases}
\end{equation}

The mean value of $\alpha_{2}$ over all monic polynomials is defined to be
%was defined in \eqref{mean} by:
\begin{equation}
\langle \alpha_{2}\rangle_{n}:=\frac{1}{q^{n}}\sum_{f\in \EM_n}\alpha_{2}(f).
\end{equation}
The generating function for the number of monic square-full polynomials of degree $n$, i.e. $\sum_{f\in M_{n}}\alpha_{2}(f) $ is
\begin{equation}
\sum_{n=0}^{\infty}\sum_{f\in M_{n}}\alpha_{2}(f) u^{n}=\frac{Z(u^{2})Z(u^{3})}{Z(u^{6})}
\end{equation}
where $Z(u)$ is the zeta function of $\fq[T]$ (also set $\zeta_{q}:=Z(q^{-s})$) , given by the following product over prime polynomials in $\fq[T]$
\begin{equation}
Z(u)=\prod_{P }(1-u^{\deg P})^{-1}=\frac{1}{1-qu}
\end{equation}
By expanding an comparing coefficients we have for $n>6$
\begin{equation}
\sum_{f\in M_{n}}\alpha_{2}(f)=\frac{q^{n/2}}{\zeta_{q}(3)}\sum_{\substack{j=n\mod 2\\ 0\leq j \leq \lfloor\dfrac{n}{3}\rfloor-2}}q^{\frac{-j}{2}}
+q^{\lfloor{\frac{n-\lfloor n/3 \rfloor}{2}\rfloor}}
\end{equation}
therefore in the limit of $q\rightarrow\infty$ we get
\begin{equation}\label{average}
\langle \alpha_{2}\rangle_{n}\sim q^{\lfloor n/2 \rfloor-n}.
\end{equation}
%%%%%%%%%%%%%%%%%%%%%%%%%%%%%%%%%%%%%%%%%%%%%%%%%%%%%%%%%%%%%%%%%%%%%%%%%%%%%%%%%%%%%%%%%%%%%%%%%%%%%%%%
\subsection{Arithmetic progressions}
%The mean value of $\alpha_{2}$ over all monic polynomials is given by:
%\begin{equation}\label{mean}
%\langle \alpha_{2}\rangle_{n}:=\frac{1}{q^{n}}\sum_{f\in \EM_n}\alpha_{2}(f).
%\end{equation}
Let $Q\in \fq[T]$ be a squarefree polynomial of a positive degree. The sum of $\alpha_{2}$ over all monic polynomials of degree $n$ lying in the arithmetic progressions $f=A\mod Q$ is
\begin{equation}\label{sum progressions}
\mathit{S}_{\alpha_{2};n;Q}(A):=\sum_{\substack{f\in \EM_n\\ f=A \mod Q}}\alpha_{2}(f)
\end{equation}
The average of this sum $\mathit{S}_{\alpha_{2};n;Q}(A)$ when we vary $A$ over residue classes coprime to $Q$ is
\begin{equation}
\langle \mathit{S}_{\alpha_{2};n;Q}\rangle=\frac{1}{\Phi(Q)}\sum_{\substack{f\in \EM_n\\ (f,Q)=1}}\alpha_{2}(f)
\end{equation}
where $\Phi(Q)$ is the number of invertible residues modulo $Q$.

In section \ref{arithmetic progressions}  we will consider the variance of $\mathit{S}_{\alpha_{2};n;Q}$ which is defined to be the average of the squared difference between $\mathit{S}_{\alpha_{2};n;Q}$ and its mean value
\begin{equation}\label{var start}
\Var(\mathit{S}_{\alpha_{2};n;Q})=\frac{1}{\Phi(Q)}\sum_{\substack{A \mod Q\\ (A,Q)=1}}|\mathit{S}_{\alpha_{2};n;Q}-\langle \mathit{S}_{\alpha_{2};n;Q}\rangle|^{2}.
\end{equation}
proving the following Theorem:

\begin{theorem}\label{main t prog}
Let $Q$ be a prime polynomial and set $N:=\deg Q - 1$, then in the limit $q\rightarrow\infty$ the following holds
\\for $N\leq n\leq 2N$
\begin{equation*}
\Var(\mathit{S}_{\alpha_{2};n;Q})\sim \frac{q^{\lfloor\frac{n}{2}\rfloor}}{\Phi(Q)}
\end{equation*}
\\for $2N< n$ even
\begin{equation*}
\Var(\mathit{S}_{\alpha_{2};n;Q})\sim \frac{q^{n}}{\Phi(Q)^{2}}
\end{equation*}
\\for $2N< n$ odd
\begin{equation*}
\Var(\mathit{S}_{\alpha_{2};n;Q})\sim \frac{q^{n-3}}{\Phi(Q)^{2}}\cdot |\sum_{j=1}^{\deg Q -1}\alpha_{j}(\chi_{2})|^{2}=O(\frac{q^{n-2}}{\Phi(Q)^{2}})
\end{equation*}
where $\chi_{2}$ is the quadratic character mod $Q$ and $\alpha_{j}(\chi_{2})$ are the inverse roots of the L-function associated to $\chi_{2}$.
\end{theorem}
Note that the restriction that $Q$ is a prime is for simplicity reasons only.
%%%%%%%%%%%%%%%%%%%%%%%%%%%%%%%%%%%%%%%%%%%%%%%%%%%%%%%%%%%%%%%%%%%%%%%%%%%%%%%%%%%%%%%%%%%%%%%%%%%%%%%%
\subsection{Short intervals}

 A ``short interval" in $\fq[x]$ is a set of the form
\begin{equation}
I(A;h) = \{f:||f-A||\leq q^h\}
\end{equation}
where  $A\in \EM_n$ and $0\leq h\leq n-2$.
%\footnote{For $h=n-1$, $I(A;n-1) = \EM_n$ is the set of all monic
%polynomials of degree $n$}
the norm is
\begin{equation}
||f||:=\#\fq[t]/(f) = q^{\deg f}\;.
\end{equation}
The cardinality of such a short interval is
\begin{equation}
\#I(A;h)=q^{h+1}:=H \;.
\end{equation}

We define for $1\leq h<n$ and $A\in \EM_n$
\begin{equation}
\EN_{\alpha_{2};h}(A):= \sum_{f\in I(A;h)}\alpha_{2}(f)
\end{equation}
to be the number of square-full polynomials in the interval $I(A;h)$.

The mean value of $\EN_{\alpha_{2};h}$ when we average over $A\in \EM_n$ is
\begin{equation}
\langle \EN_{\alpha_{2};h} \rangle := \frac{1}{q^{n}} \sum_{A\in \EM_n} \EN_{\alpha_{2};h}(A)
\end{equation}

In section \ref{intervals} we will compute the variance of $\EN_{\alpha_{2};h}$
\begin{equation}
\Var(\EN_{\alpha_{2};h}):=\frac{1}{q^{n}}\sum_{A\in \EM_n}|\EN_{\alpha_{2};h}(A)-\langle \EN_{\alpha_{2};h} \rangle|^{2}
\end{equation}
proving the following Theorem:
\begin{theorem}\label{main t intervals}
Set $N:=n-h-2$, then in the limit $q\rightarrow\infty$ the following holds
\\for $0\leq n\leq 2N$
\begin{equation}
\Var(\EN_{\alpha_{2};h})\sim \frac{H}{q^n}\cdot q^{\lfloor\frac{n}{2}\rfloor}
\end{equation}
\\for $2N< n\leq 5N$
\begin{equation}\label{second1}
\Var(\EN_{\alpha_{2};h})\sim \frac{H}{q^n}\cdot q^{\lfloor\frac{n+N}{3}\rfloor}
\end{equation}
\\for $5N< n$
\begin{equation}\label{third1}
\Var(\EN_{\alpha_{2};h})\sim \frac{H}{q^n}\cdot q^{\frac{n+N}{6}}\cdot q^{\frac{-\lambda_{n}}{6}}
\end{equation}
where
\begin{equation}\label{third2}
\lambda_{n}= \begin{cases}
0 & n=5N \mod 6
\\ 7 & n=5N+1 \mod 6
\\6 & n=5N+2 \mod 6
\\3 & n=5N+3 \mod 6
\\4 & n=5N+4 \mod 6
\\11 & n=5N+5 \mod 6
\end{cases}
\end{equation}
\end{theorem}

%%%%%%%%%%%%%%%%%%%%%%%%%%%%%%%%%%%%%%%%%%%%%%%%%%%%%%%%%%%%%%%%%%%%%%%%%%%%%%%%%%%%%%%%%%%%%%%%%%%%

\section{Dirichlet characters and Katz's equidistribution results}\label{dirichlet characters and Katz's equidistribution results}
\subsection{Dirichlet characters}\label{Dirichlet characters}
Let $Q(T)\in \fq[T]$ be a polynomial of positive degree. A Dirichlet character modulo Q is a homomorphism
\begin{equation}
\chi: (\fq[T]/(Q))^{\times}\rightarrow \mathbb{C}^{\times}.
\end{equation}
One can extend $\chi$ to $\fq[T]$ by defining it to vanish on polynomials which are not coprime to $Q$. We denote by $\Gamma(Q) $ the group of all Dirichlet characters modulo $Q.$
Note that $|\Gamma(Q)|$ is the Euler totient function $\Phi(Q)$.
A Dirichlet character needs then to satisfy the following:
$\chi(fg)=\chi(f)\chi(g)$ for all $f,g \in \fq[T]$, $\chi(1)=1$ and $\chi(f+hQ)=\chi(f)$ for all $f,h \in \fq[T].$

The orthogonality relations for Dirichlet characters are:
\begin{equation}\label{first o}
\frac{1}{\Phi(Q)}\sum_{\chi \mod Q}\bar{\chi}(A)\chi(N)=\begin{cases} 1 & N=A\mod Q
 \\ 0 & \mbox{otherwise} \end{cases}
\end{equation}

\begin{equation}\label{second o}
\frac{1}{\Phi(Q)}\sum_{A \mod Q}\bar{\chi_{1}}(A)\chi_{2}(A)=\begin{cases} 1 & \chi_{1}=\chi_{2}
 \\ 0 & \mbox{otherwise} \end{cases}
\end{equation}

A character $\chi$ is called "even" if it acts trivially on the elements of $\fq$, i.e. if $\chi(cf)=\chi(f)$ for all $0\neq c\in \fq.$ Therefore, the number of even characters is given by $\Phi^{ev}(Q)=\frac{\Phi(Q)}{q-1}$. For example there are $q^{m-1}$ even character mod $T^{m}$. Out of this there are $O(q^{m-2})$ nonprimitive even characters mod $T^{m}$ (see subsection 3.3 in \cite{KR}) . A character is called "odd" if it is not even.

A character $\chi$ is \textit{primitive} if there is no proper divisor $Q'|Q$ such that $\chi(f)=1$ whenever $f$ is co-prime to $Q$ and $f=1\mod Q'.$
%We denote by $\Phi_{prim}(Q)$ the set of all primitive characters mod $Q$, and by
%$\Phi_{prim}^{ev}(Q)$ and $\Phi_{prim}^{odd}(Q)$ the set of primitive even characters mod $Q$ and %primitive odd characters mod $Q$ respectively. We also denote by $\Phi_{d-prim}(Q)$ the set of all %characters $\chi$ mod $Q$ such that $\chi^{d}$ is primitive.
Define the following:
\\
\\
$\Gamma_{prim}(Q)$-the set of all primitive characters mod $Q$, $|\Gamma_{prim}(Q)|:=\Phi_{prim}(Q).$
\\
\\
$\Gamma_{prim}^{ev}(Q)$-the set of primitive even characters mod $Q$, $|\Gamma_{prim}^{ev}(Q)|:=\Phi_{prim}^{ev}(Q).$
\\
\\
$\Gamma_{prim}^{odd}(Q)$-the set of primitive odd characters mod $Q$, $|\Gamma_{prim}^{odd}(Q)|:=\Phi_{prim}^{odd}(Q).$
\\
\\
$\Gamma_{d-prim}(Q)$-the set of all characters $\chi$ mod $Q$ such that $\chi^{d}$ is primitive for the fixed integer $d>0$, $|\Gamma_{d-prim}(Q)|:=\Phi_{d-prim}(Q).$ note that this is a subset of $\Gamma_{prim}(Q).$
\\
\\
$\Gamma_{d-prim}^{d-odd}(Q)$-the set of all characters $\chi$ mod $Q$ such that $\chi^{d}$ is primitive and odd for the fixed integer $d>0$, $|\Gamma_{d-prim}^{d-odd}(Q)|:=\Phi_{d-prim}^{d-odd}(Q).$
\\
\\
\\Next, we will check the proportion of the set $\Gamma_{d-prim}(Q)$ in the the group of all characters mod $Q$ ,$\Gamma(Q).$
\begin{lemma}
Let $Q\in \fq[t]$ be a square-free polynomial, then in the limit of a large field size $q\rightarrow\infty$,
\begin{equation}\label{d-prim}
\frac{\Phi_{d-prim}(Q)}{\Phi(Q)}=1+O(\frac{1}{q})
\end{equation}
\end{lemma}

\begin{proof}
Consider equation (3.21) in \cite{KR}. It asserts that as $q\rightarrow\infty$, almost all characters are \textbf{primitive} in the sense that
\begin{equation}\label{prim}
\frac{\Phi_{prim}(Q)}{\Phi(Q)}=1+O(\frac{1}{q})
\end{equation}
We claim, in the lemma above, that this also holds when $\Phi_{prim}(Q)$ is replaced by $\Phi_{d-prim}(Q)$.
And indeed, if $Q$ factors into k irreducible polynomials $P_{i}$ of degree $d_{i}$,  $\deg Q:=n=\sum_{i=1}^{k}d_{i}$ and $\Phi(Q)=\prod_{i=1}^{k}(q^{d_{i}}-1)$. Therefore, $\Phi(Q) \sim q^{n}$ as $q\rightarrow\infty.$ By the above combined with \eqref{prim}, we may conclude that the number of non primitive characters is $O(q^{n-1}).$
Note that a character $\chi$ does not lie in $\Gamma_{d-prim}(Q)$ if $\chi^{d}$ does not lie in $\Gamma_{prim}(Q)$. In that case, $\chi^{d}$ lies in $\frac{\Gamma(Q)}{\Gamma_{prim}(Q)}$, and recall that $\#(\frac{\Gamma(Q)}{\Gamma_{prim}(Q)})=O(q^{n-1}).$
Now consider the map $\chi \mapsto \chi^{d}$. Its kernel is of cardinality $\#\lbrace \chi | \chi^{d}=1 \rbrace$. Since every $\chi$ is a product of characters $\chi_{j}$ of $\fq[T]/(f_{j})$ when $Q=\prod_{j=1}^{k}f_{j},$ $\deg(f_{j})=d_{j}$ and $\sum_{i=1}^{k}d_{i}=n$, the following bound holds:
\begin{equation}
\# \lbrace \chi | \chi^{d}=1 \rbrace\leq d^{n}.
\end{equation}
Since every character $\chi$ such that $\chi^{d}$ is not primitive can be written as a product of a non primitive character and an element of the kernel, we get that $\#(\frac{\Gamma(Q)}{\Gamma_{d-prim}(Q)})$ can be at most $d^{n}\times O(q^{n-1})=O(q^{n-1})$, therefore we get
$\eqref{d-prim}.$
\end{proof}

Note that equation (3.25) in \cite{KR}  asserts that as $q\rightarrow\infty$, almost all characters are \textbf{primitive} and \textbf{odd} in the sense that
\begin{equation}\label{primodd}
\frac{\Phi_{prim}^{odd}(Q)}{\Phi(Q)}=1+O(\frac{1}{q})
\end{equation}
Hence, exactly as before, we may also show that
\begin{lemma}\label{d-primodd lemma}
Let $Q\in \fq[t]$ be a square-free polynomial, then in the limit of a large field size $q\rightarrow\infty$,
\begin{equation}\label{d-primodd}
\frac{\Phi_{d-prim}^{d-odd}(Q)}{\Phi(Q)}=1+O(\frac{1}{q})
\end{equation}
\end{lemma}

Next, we will prove short lemma stating that under certain restrictions on the characteristic of the field, the primitivity of $\chi$ and $\chi^{d}$ is equivalent when $\chi$ is an even character mod $T^{m}$. This lemma will be useful later on in section \ref{intervals}.

\begin{lemma}\label{chi d prim}
%Let $\chi$ be a Dirichlet character mod $Q$,
Let $d$ be an integer co-prime to $\Phi(Q)$. Then the map $\chi\mapsto \chi^{d}$ is an automorphism of the group of characters mod $Q$, i.e. an automorphism of $\Gamma(Q)$.
\end{lemma}
\begin{proof}
The map is clearly an endomorphism since the group is abelian. Now, $d$ is co-prime to the order of the group therefore there aren't any elements who's order dividing $d$ and hence the kernel of the map is trivial.
\end{proof}

\begin{lemma}\label{chi d prim}
Let $\chi$ be an even Dirichlet character mod $T^{m}$, and let $d$ be an integer s.t. $d<p$ when $p$ is the characteristic of the field $\fq$. Then $\chi$ is a primitive character if and only if $\chi^{d}$ is a primitive character.
\end{lemma}
\begin{proof}
The order of the the subgroup of even characters mod $T^{m}$ is $\Phi^{ev}(T^{m})=q^{m-1}$. Taking $d<p$ when $p$ is the characteristic of the field $\fq$
gives $d$ co-prime to $\Phi(T^{m})$, in which case the above lemma applies.
\end{proof}

\subsection{Dirichlet L-functions}\label{Dirichlet L-functions}
The L-function associated to a dirichlet character $\chi\mod Q$ is defined as the following product over prime polynomials $P\in \fq[T]$
\begin{equation}
L(u,\chi)=\prod_{P \nmid Q}(1-\chi(P)u^{\deg P})^{-1}
\end{equation}
The product is absolutely convergent for $|u|<1/q$. For $\chi=\chi_{0}$ the trivial character mod $Q$
\begin{equation}
L(u,\chi_{0})=Z(u)\prod_{P | Q}(1-u^{\deg P})
\end{equation}
If $Q\in \fq[T]$ is a polynomial of degree $\deg Q\geq 2$ and $\chi$ is a nontrivial character mod $Q$, then the L-function associated to $\chi$ i.e. $L(u,\chi)$ is a polynomial in $u$ of degree $\deg Q -1$. If $\chi$ is even then $L(u,\chi)$ has a trivial zero at $u=1$. Now, we may factor $L(u,\chi)$ in terms of the inverse roots
\begin{equation}
L(u,\chi)=\prod_{j=1}^{\deg Q -1}(1-\alpha_{j}(\chi)u)
\end{equation}
for which the Riemann Hypothesis, proved by Weil, asserts that for each (nonzero) inverse root, either $\alpha_{j}(\chi)=1$ or
\begin{equation}\label{weil}
|\alpha_{j}(\chi)|=q^{1/2}
\end{equation}
For a primitive odd character mod $Q$ all the inverse roots have absolute value $q^{1/2}$. For a primitive even character mod $Q$ all the inverse roots have absolute value $q^{1/2}$, except for the trivial zero at $1$. Thus, we may write $\alpha_{j}(\chi)=q^{1/2}\ee^{i\Theta_{j}}$, and the L-function (for a primitive character $\chi$) is
\begin{equation}
L(u,\chi)=(1-\lambda_{\chi}u)^{-1}\det(I-uq^{1/2}\Theta_{\chi}),\quad \quad \Theta_{\chi}=\diag(\ee^{i\Theta_{1}},\ldots,\ee^{i\Theta_{N}}).
\end{equation}
where $N=\deg Q -1$ and $\lambda_{\chi}=0$ for odd character $\chi$. For even $\chi$ we have $N=\deg Q -2$ and $\lambda_{\chi}=1$.
The unitary matrix $\Theta_{\chi}\in U(N)$ determines a unique conjugacy class which is called the unitarized Frobenius matrix of $\chi$.

\subsection{Katz's equidistribution results}\label{Katz's equidistribution results}

The main ingredients in our results on the variance are equidistribution and independence results for the
Frobenii $\Theta_{\chi}$ due to N. Katz.
\begin{theorem}\label{k-e-first}\cite{K2}
Fix $m \geq 4$. The unitarized Frobenii $\Theta_{\chi}$ for the
family of even primitive characters $\mod T^{m+1}$ become equidistributed in the
projective unitary group $PU(m-1)$ of size $m-1$, as $q$ goes to infinity.
\end{theorem}
\begin{theorem}\label{k-e-second}\cite{K4}
If $m \geq 5$ and in addition the characteristics of the fields $\fq$ is bigger than $13$, then the set of conjugacy classes $(\Theta_{\chi^{2}}, \Theta_{\chi^{3}}, \Theta_{\chi^{6}} )$ become equidistributed in the space of conjugacy classes of the product $PU(m-1) \times PU(m-1)\times PU(m-1)$ as $q$ goes to infinity.
\end{theorem}

For odd characters, the corresponding equidistribution and independence
results are

\begin{theorem}\label{k-o-first}\cite{K1}
Fix $m \geq 2$. Suppose we are given a sequence of finite
fields $\fq$ and squarefree polynomials $Q(T) \in \fq[T]$ of degree $m$. As $q \rightarrow \infty$,
the conjugacy classes $\Theta_{\chi}$ with $\chi$ running over all primitive odd characters
modulo $Q$, are uniformly distributed in the unitary group $U(m-1)$.
\end{theorem}

\begin{theorem}\label{k-o-second}\cite{K3}
If in addition we restrict the characteristics of the fields $\fq$ is bigger than $6$, then the set of conjugacy classes $(\Theta_{\chi^{2}}, \Theta_{\chi^{3}}, \Theta_{\chi^{6}} )$ with $\chi$ running over all characters such that $\chi^{2},\chi^{3},\chi^{6}$ are primitive odd characters
modulo $Q$,become equidistributed in the space of conjugacy classes of
the product $U(m-1) \times U(m-1)\times U(m-1)$ as $q$ goes to infinity.
\end{theorem}

%%%%%%%%%%%%%%%%%%%%%%%%%%%%%%%%%%%%%%%%%%%%%%%%%%%%%%%%%%%%%%%%%%%%%%%%%%%%%%%%%%%%%%%%%%%%%%%%%%%

\section{The variance in arithmetic progressions}\label{arithmetic progressions}

\subsection{The mean value}
Given a polynomial $Q\in \fq[T]$ the average of $\mathit{S}_{\alpha_{2};n;Q}(A)$ when we vary $A$ over residue classes co-prime to $Q$ (see \eqref{sum progressions}) equals to
\begin{equation}
\langle \mathit{S}_{\alpha_{2};n;Q}\rangle=\frac{1}{\Phi(Q)}\sum_{\substack{f\in \EM_n\\ (f,Q)=1}}\alpha_{2}(f)=\frac{1}{\Phi(Q)}\sum_{f\in \EM_n}\chi_{0}(f)\alpha_{2}(f)
\end{equation}
To evaluate this consider the generating function
\begin{equation}
\begin{split}
\langle \mathit{S}_{\alpha_{2};n;Q}\rangle&=\frac{1}{\Phi(Q)}\sum_{\substack{f\in \EM_n\\ (f,Q)=1}}\alpha_{2}(f)
\\&=\frac{1}{\Phi(Q)}\cdot\sum_{f\in \EM_n}\chi_{0}(f)\alpha_{2}(f)
\\&=\frac{1}{\Phi(Q)}\cdot\frac{L(u^{2},\chi_{0})L(u^{3},\chi_{0})}{L(u^{6},\chi_{0})}
\\&=\frac{1}{\Phi(Q)}\cdot\frac{Z(u^{2})Z(u^{3})}{Z(u^{6})}\prod_{P|Q}(\frac{(1-u^{2\deg P})(1-u^{3\deg P})}{(1-u^{6\deg P})})
\\&=\frac{1}{\Phi(Q)}\cdot\frac{Z(u^{2})Z(u^{3})}{Z(u^{6})}\prod_{P|Q}(\frac{(1-u^{2\deg P})}{(1+u^{3\deg P})})
\end{split}
\end{equation}
By expanding and comparing coefficients it is clear that the leading order coefficient in $q$ (we are interested in $q\rightarrow\infty$) comes from $\frac{Z(u^{2})Z(u^{3})}{Z(u^{6})}.$ Therefor, by \ref{average}, we have
\begin{equation}\label{mean value progressions}
 \langle \mathit{S}_{\alpha_{2};n;Q}\rangle\sim \frac{q^{\lfloor n/2 \rfloor}}{\Phi(Q)}.
\end{equation}
In the rest of this section we will evaluate the variance of $\mathit{S}_{\alpha_{2};n;Q}$ i.e. the average of the squared difference between $\mathit{S}_{\alpha_{2};n;Q}$ and its mean value.

\subsection{The case of small n}
See also subsection 4.2 in \cite{KRsf}. If $n<\deg Q$ then there is at most one polynomial $f\in \fq[T]$ of degree $n$ such that $f=A\mod Q$. In this case, when $q \rightarrow\infty$
\begin{equation}
Var(\mathit{S}_{\alpha_{2};n;Q})\sim\frac{q^{n}}{\Phi(Q)}\langle \alpha_{2}\rangle_{n}
\end{equation}
Indeed, if $n<\deg Q$
we can use \eqref{mean value progressions} to see that
\begin{equation}
|\langle \mathit{S}_{\alpha_{2};n;Q}\rangle|\ll_{n}\frac{1}{q^{n/2}}
\end{equation}
Hence
\begin{equation}\label{small var}
\begin{split}
   Var(\mathit{S}_{\alpha_{2};n;Q})= & \frac{1}{\Phi(Q)}\sum_{\substack{A \mod Q\\ (A,Q)=1}} |\mathit{S}_{\alpha_{2};n;Q}(A)|^{2}(1+O(\frac{1}{q^{n/2}}))\\
     & = \frac{1}{\Phi(Q)}\sum_{\substack{f\in \EM_n\\ (f,Q)=1}}|\alpha_{2}(f)|^{2}(1+O(\frac{1}{q^{n/2}}))\\
     & \sim\frac{q^{n}}{\Phi(Q)}\langle \alpha_{2}\rangle_{n}\\
     &  \sim\frac{q^{\lfloor n/2 \rfloor}}{\Phi(Q)}
\end{split}
\end{equation}

\subsection{A formula for the variance}
We present a formula for the variance of $\mathit{S}_{\alpha_{2};n;Q}$ using Dirichlet characters \cite[\S4.1]{KRsf}.
We start with  the following expansion, using the first orthogonality
relation for Dirichlet characters (see \eqref{first o}) to pick out an arithmetic
progression:
\begin{equation}\label{expand ES}
\mathit{S}_{\alpha_{2},n,Q}(A) = \frac 1{\Phi(Q)} \sum_{\substack{f\in \mathcal M_n\\
 (f,Q)=1}} \alpha_{2}(f)  + \frac{1}{\Phi(Q)} \sum_{\chi\neq
\chi_0}\overline{\chi(A)} \EM(n;\alpha_{2}\chi)
\end{equation}
Where
\begin{equation}\label{def of M(n)}
\EM(n;\alpha_{2}\chi) := \sum_{f\in M_{n}}\alpha_{2}(f)\chi(f)
\;.
\end{equation}
Note that the contribution of the trivial character equals to the average of $\mathit{S}_{\alpha_{2};n;Q}(A)$. Therefore
\begin{equation}
\mathit{S}_{\alpha_{2};n;Q}-\langle \mathit{S}_{\alpha_{2};n;Q}\rangle=
\frac{1}{\Phi(Q)} \sum_{\chi\neq \chi_{0} \mod Q}\overline{\chi}(A)\EM(n;\alpha_{2}\chi)
\end{equation}
Using the above and the second orthogonality relation for Dirichlet characters (see \eqref{second o}), we have the following expression for the variance:
\begin{equation}\label{varar}
Var(\mathit{S}_{\alpha_{2};n;Q})=\langle |\mathit{S}_{\alpha_{2};n;Q}-\langle \mathit{S}_{\alpha_{2};n;Q}\rangle|^{2}\rangle=\frac{1}{\Phi(Q)^{2}}\sum_{\chi\neq\chi_{0}}|\EM(n;\alpha_{2}\chi)|^{2}
\end{equation}

\subsection{The quadratic character and the cubic character}\label{The quadratic character and the cubic character}
Next, we evaluate $\EM(n;\alpha_{2}\chi)$ for $\chi=\chi_{2}$ a quadratic character and for $\chi=\chi_{3}$ a cubic character.
We assume here for simplicity that $Q$ is prime polynomial.
The sum $\EM(n;\alpha_{2}\chi)$ given by \eqref{def of M(n)}, is the coefficient of $u^{n}$ in the expansion of $\frac{L(u^{2},\chi^{2})L(u^{3},\chi^{3})}{L(u^{6},\chi^{6})}.$
Thus for $\chi=\chi_{2}$, the generating function of $\EM(n;\alpha_{2}\chi_{2})$ has the following form
\begin{equation}\label{quadratic generating}
  L(u^{3},\chi_{2})\cdot\frac{Z(u^{2})(1-u^{2\deg Q})}{Z(u^{6})(1-u^{6\deg Q})}
\end{equation}
by expanding and comparing coefficients while bearing in mind the Riemann hypothesis \eqref{weil}, we get
\begin{equation}\label{M quadratic}
\EM(n;\alpha_{2}\chi_{2})\sim
\begin{cases}
 q^{\frac{n}{2}} & n \even
\\-q^{\frac{n-3}{2}} \cdot \sum_{j=1}^{\deg Q -1}\alpha_{j}(\chi_{2}) & n \odd
\end{cases}
\end{equation}
where $\alpha_{j}(\chi_{2})$ are the inverse roots of $L(u^{3},\chi_{2})$.

For $\chi=\chi_{3}$, the generating function of $\EM(n;\alpha_{2}\chi_{3})$ has the following form
\begin{equation}\label{quadratic generating}
  L(u^{2},\chi_{3}^{2})\cdot\frac{Z(u^{3})(1-u^{3\deg Q})}{Z(u^{6})(1-u^{6\deg Q})}
\end{equation}
as before we get
\begin{equation}\label{M quadratic}
\EM(n;\alpha_{2}\chi_{3})\sim
\begin{cases}
 q^{\frac{n}{3}} & n =0 \mod 3
\\q^{\frac{n-4}{3}} \cdot \sum_{j,l=1}^{\deg Q -1}\alpha_{j}(\chi_{3})\alpha_{l}(\chi_{3})  & n=1 \mod 3
\\-q^{\frac{n-2}{3}} \cdot \sum_{j=1}^{\deg Q -1}\alpha_{j}(\chi_{3})  & n=2 \mod 3
\end{cases}
\end{equation}
where $\alpha_{j}(\chi_{3})$ are the inverse roots of $L(u^{2},\chi_{3})$.

\subsection{Average of the sum $\EM(n;\alpha_{2}\chi)$}

\begin{lemma}
For a dirichlet character $\chi$ mod $Q$, such that $\chi^{2},\chi^{3},\chi^{6}$ are odd and primitive
\begin{equation}\label{coeff}
\EM(n;\alpha_{2}\chi)=\sum_{\substack{2j+3l+6k=n\\0\leq j \leq N\\ 0\leq l \leq N\\ 0\leq k}}q^{\frac{j+k+l}{2}}\tr\Lambda_{j}(\Theta_{\chi^{2}})\tr \Lambda_{l}(\Theta_{\chi^{3}})\tr \sym^{k}(\Theta_{\chi^{6}})
\end{equation}
Where $\Theta_{\chi}\in U(\deg Q-1)$ is the unitarized Frobenius matrix, $\sym^{n}$ is the symmetric $n-th$ power representation, and $\Lambda_{n}$ is the exterior $n-th$ power representation.
\\For $\chi\neq \chi_{0},\chi_{2},\chi_{3}$ mod $Q$ not of the above, the following bound holds
\begin{equation}\label{bound_M}
|\EM(n;\alpha_{2}\chi)|\ll_{\deg Q} \sum_{\substack{2j+3l+6k=n\\0\leq j \leq N\\ 0\leq l \leq N\\ 0\leq k}}q^{\frac{j+k+l}{2}}
\end{equation}
\end{lemma}
\begin{proof}
The sum $\EM(n;\alpha_{2}\chi)$ given by \eqref{def of M(n)}, is the coefficient of $u^{n}$ in the expansion of $\frac{L(u^{2},\chi^{2})L(u^{3},\chi^{3})}{L(u^{6},\chi^{6})}.$
For an odd primitive characters $\chi$ mod $Q$, we use the Riemann Hypothesis \eqref{weil} (Weil’s theorem) to write
\begin{equation}
L(u,\chi)=\det(I-uq^{\frac{1}{2}}\Theta_{\chi})=\sum_{i=0}^{\deg Q-1}u^{i}q^{i/2}\tr \Lambda_{i}(\Theta_{\chi})
\end{equation}
and
\begin{equation}
\frac{1}{L(u,\chi)}=\frac{1}{\det(I-uq^{\frac{1}{2}}\Theta_{\chi})}=\sum_{i=0}^{\infty}u^{i}q^{i/2}\tr \sym^{i}(\Theta_{\chi})
\end{equation}
Abbreviate as follows:
$$\Lambda_{i}(\chi):=\tr \Lambda_{i}(\Theta_{\chi}),~~~~~~~\sym^{i}(\chi)=\tr \sym^{i}(\Theta_{\chi})$$
to have
\begin{equation}
\frac{L(u^{2},\chi^{2})L(u^{3},\chi^{3})}{L(u^{6},\chi^{6})}=
\sum_{j=0}^{\deg Q-1}\sum_{l=0}^{\deg Q-1}\sum_{k=0}^{\infty}u^{2j+3l+6k}q^{\frac{j+k+l}{2}}\Lambda_{j}(\chi^{2})\Lambda_{l}(\chi^{3})Sym^{k}(\chi^{6})
\end{equation}
Hence the coefficient of $u^{n}$ is indeed given by \eqref{coeff}.

For $\chi\neq \chi_{0},\chi_{2},\chi_{3}$ mod $Q$  for which at least one of  $\chi^{2},\chi^{3},\chi^{6}$ is not primitive or not odd, we still have $L(u,\chi)=\prod_{j=1}^{\deg Q-1}(1-\alpha_{j}(\chi)u)$ with all the inverse roots $|\alpha_{j}(\chi)|=q^{\frac{1}{2}}$ or $|\alpha_{j}(\chi)|=1$ , and hence we obtain \eqref{bound_M}.
\end{proof}

Next, we want to evaluate $\sum_{\substack{2j+3l+6k=n\\0\leq j \leq N\\ 0\leq l \leq N\\ 0\leq k}}q^{j+k+l}$. Denote
\begin{equation}\label{s(n)}
 \textit{S}(n):=\sum_{\substack{2j+3l+6k=n\\0\leq j \leq N\\ 0\leq l \leq N\\ 0\leq k}}q^{j+k+l}
\end{equation}
\begin{lemma}\label{eval s(n)}
Let $N:=\deg Q - 1$ then in the limit $q\rightarrow\infty$,
\\for $0\leq n\leq 2N$
\begin{equation}
\textit{S}(n)\sim q^{\lfloor\frac{n}{2}\rfloor}
\end{equation}
\\for $2N< n\leq 5N$
\begin{equation}\label{second1}
\textit{S}(n)\sim q^{\lfloor\frac{n+N}{3}\rfloor}
\end{equation}
\\for $5N< n$
\begin{equation}\label{third1}
\textit{S}(n)\sim q^{\frac{n+N}{6}}\cdot q^{\frac{-\lambda_{n}}{6}}
\end{equation}
where
\begin{equation}\label{third2}
\lambda_{n}= \begin{cases}
0 & n=5N \mod 6
\\ 7 & n=5N+1 \mod 6
\\6 & n=5N+2 \mod 6
\\3 & n=5N+3 \mod 6
\\4 & n=5N+4 \mod 6
\\11 & n=5N+5 \mod 6
\end{cases}
\end{equation}
\end{lemma}

\begin{proof}
In order to find the leading order term of $\textit{S}(n)$ we need first to take the maximal possible $j$ and then the maximal possible $l$ which satisfy $2j+3l+6k=n, 0\leq j \leq N, 0\leq l \leq N$ (note that $k$ will then be determined).

In the first case $0\leq n \leq 2N$ if $2$ divides $n$ then $j=n/2,l=0,k=0$ will clearly give the leading order term. If $2$ does not divide $n$ then $j=(n-3)/2,l=1,k=0$ will give the leading order term.

In the second case $2N< n \leq 5N$, we write $j=N-i_{j}$ and then we have $n-2N=6k+3l-2i_{j}$ and so clearly the values for $k,l,i_{j}$ that will give the leading order term, depend on the value of $n-2N \mod 3$ (or equivalently $n+N \mod 3$). Here our first priority is to minimize $i_{j}$ and then to maximize $l$. The leading order term will be given by $q^{\frac{n+N-i_{j}}{3}-\frac{k}{2}}$ which gives $q^{\lfloor\frac{n+N}{3}\rfloor}.$

In the last case $5N< n$, we write $j=N-i_{j}$ and $l=N-i_{l}$, then we have $n-5N=6k-3i_{l}-2i_{j}$ ans so clearly the values for $k,i_{l},i_{j}$ that will give the leading order term, depend on the value of $n-5N \mod 6$. Here our first priority is to minimize $i_{j}$ and then to minimize $i_{l}$. The leading order term will be given by $q^{\frac{n+7N-3i_{l}-4i_{j}}{6}}$. Note that in the notations of \eqref{third1} and \eqref{third2} we have $4i_{j}+3i_{l}=\lambda_{n}$

\end{proof}

\subsection{Proof of Theorem \ref{main t prog}}
Recall that by Lemma \ref{d-primodd lemma}, we have that the number of characters which are not in $\Gamma_{d-prim}^{d-odd}(Q)$ is $O(\frac{\Phi(Q)}{q})$ when $q\rightarrow\infty.$ Therefore, by using the formula for the variance that was given in \eqref{varar} we may write
\begin{equation}
\begin{split}
                   Var(\mathit{S}_{\alpha_{2};n;Q})&=\frac{1}{\Phi(Q)^{2}}\sum_{\substack{\chi\neq\chi_{0} \\
                   \chi \in \Gamma_{6-prim}^{6-odd}(Q)}}|\EM(n;\alpha_{2}\chi)|^{2} +\frac{1}{\Phi(Q)^{2}}|\EM(n;\alpha_{2}\chi_{2})|^{2}\\
                     &  +\lambda_{3}\frac{1}{\Phi(Q)^{2}}|\EM(n;\alpha_{2}\chi_{3})|^{2}
                    + O(\frac{\textit{S}(n)}{\Phi(Q)q})
                \end{split}
\end{equation}
Where $\lambda_{3}=2$ if $|Q|=1 \mod 3$, and zero otherwise. Note that we assume $Q$ is prime and that the characteristic of the field $\fq$ is odd therefor there is one quadratic character mod $Q$ and either two or zero cubic characters, depending on if $|Q|=1 \mod 3$ or not.
\\For the first sum, use equation \eqref{coeff} to have
\begin{equation}
\frac{1}{\Phi(Q)^{2}}\sum_{\substack{\chi\neq\chi_{0} \\
\chi \in \Gamma_{6-prim}^{6-odd}(Q)}}|\sum_{\substack{2j+3l+6k=n\\0\leq j \leq N\\ 0\leq l \leq N\\ 0\leq k}}q^{\frac{j+k+l}{2}}\Lambda_{j}(\chi^{2})\Lambda_{l}(\chi^{3})Sym^{k}(\chi^{6})|^2
\end{equation}
We can use now the equidistribution result given in Theorem \ref{k-o-second}, to have
\begin{equation}
\frac{1}{\Phi(Q)}\iiint_{U(N)}|\sum_{\substack{2j+3l+6k=n\\0\leq j \leq N\\ 0\leq l \leq N\\ 0\leq k}}q^{\frac{j+k+l}{2}}\tr \Lambda_{j}(U_{1})\tr \Lambda_{l}(U_{2})\tr \sym^{k}(U_{3})|^2 \,dU_{1}\,dU_{2}\,dU_{3}
\end{equation}
It is well known that $\Lambda_{j}$ and $Sym^{j}$ are distinct irreducible representations of the
unitary group $U(N)$, and hence one gets
\begin{equation}\label{ort exterior}
\int_{U(N)}\tr\Lambda_{j}(U)\overline{\tr\Lambda_{i}(U)}dU=\delta_{j,i}
\end{equation}
and
\begin{equation}\label{ort symmetric}
\int_{U(N)}\tr\sym^{j}(U)\overline{\tr\sym^{i}(U)}dU=\delta_{j,i}
\end{equation}
Therefore
\begin{equation}
\frac{1}{\Phi(Q)^{2}}\sum_{\substack{\chi\neq\chi_{0} \\
                   \chi \in \Gamma_{6-prim}^{6-odd}(Q)}}|\EM(n;\alpha_{2}\chi)|^{2} \sim\frac{1}{\Phi(Q)}\sum_{\substack{2j+3l+6k=n\\0\leq j \leq N\\ 0\leq l \leq N\\ 0\leq k}}q^{j+k+l}
\end{equation}
The contribution from the second and third summands was evaluated in subsection \ref{The quadratic character and the cubic character}. Adding up everything and checking for the leading order terms by using Lemma \ref{eval s(n)} finishes the proof.
%%%%%%%%%%%%%%%%%%%%%%%%%%%%%%%%%%%%%%%%%%%%%%%%%%%%%%%%%%%%%%%%%%%%%%%%%%%%%%%%%%%%%%%%%%%%%%%%%%%%

\section{The variance over short intervals} \label{intervals}
%%%%%%%%%%%%%%%%%%%%%%%%%%%%%%%%%%%%%%%%%%%%%%%%%%%%%%%%%%%%%%%%%%%%%%%%%%%%%%%%%%%%%%%%%%%%%%%%%%%%%%%%%%%%%%%%%%%%%%%%%%%%%%%%%%%
\subsection{The mean value}
The mean value of $\EN_{\alpha_{2};h}$ when we average over $A\in \EM_n$ is
\begin{equation}
\begin{split}
\langle \EN_{\alpha_{2};h} \rangle  & := \frac{1}{q^{n}} \sum_{A\in \EM_n} \EN_{\alpha_{2};h}(A)
\\&=\frac{1}{q^{n}} \sum_{A\in \EM_n}\sum_{f\in I(A;h)}\alpha_{2}(f)
\\&=q^{h+1}\frac{1}{q^{n}} \sum_{f\in \EM_n}\alpha_{2}(f)
\\ &=q^{h+1}\langle \alpha_{2} \rangle_{n}
\end{split}
\end{equation}
By \eqref{average} we have in the limit $q\rightarrow\infty$
\begin{equation}\label{mean intervals}
  \langle \EN_{\alpha_{2};h} \rangle\sim H\cdot q^{\lfloor n/2 \rfloor-n}
\end{equation}
In the rest of this section we will evaluate the variance of $\EN_{\alpha_{2};h}$ i.e. the
average of the squared difference between $\EN_{\alpha_{2};h}$ and its mean value.
%%%%%%%%%%%%%%%%%%%%%%%%%%%%%%%%%%%%%%%%%%%%%%%%%%%%%%%%%%%%%%%%%%%%%%%%%%%%%%%%%%%%%%%%%%%%%%%%%%%%%%%%%%%%%%%%%%%%%%%%%%%%%%%

\subsection{An expression for the variance}
To begin the proof of Theorem~\ref{main t intervals}, we express the
variance of the short interval sums $\EN_{\alpha_{2};h}$ in terms of
sums of the function $\alpha_{2}$, twisted by primitive even Dirichlet
characters, similarly to what was done in the previous section.
\begin{lemma}
As $q \to \infty$
\begin{equation}\label{begin var intervals}
\Var(\EN_{\alpha_{2};h})    = \frac{H}{q^n}
\frac{1}{\Phi^{ev}(T^{n-h})} \sum_{\substack{\chi \bmod
T^{n-h}\\\chi\neq\chi_{0} \;{\rm even }}} \sum_{\substack{ m_{1},m_{2}=0\\m_{1},m_{2}\neq n-1}}^{n}\EM(m_{1};\alpha_{2}\chi)\overline{\EM(m_{2};\alpha_{2}\chi)}
\end{equation}
where the definition of $\EM(n;\alpha_{2}\chi)$ was first given in \eqref{def of M(n)}
\begin{equation}
\EM(n;\alpha_{2}\chi) := \sum_{f\in M_{n}}\alpha_{2}(f)\chi(f)
\;.
\end{equation}
\end{lemma}
\begin{proof}
 To compute the variance, we use
\cite[Lemma 5.4]{KRsf}  which gives an expression for the variance
of sums over short intervals of certain arithmetic functions
$\alpha$ which are ``even" ($\alpha(cf)=\alpha(f)$ for $c\in
\fq^\times$), multiplicative, and symmetric under the  map
$f^*(t):=t^{\deg f}f(\frac 1t)$, in the sense that
\begin{equation}\label{inv under*}
\alpha(f^*) = \alpha(f), \quad {\rm if}\;  f(0)\neq 0 \;.
\end{equation}
 Since the indicator function for square full polynomials $\alpha_{2}$ clearly satisfies all of
these conditions, we may use  \cite[Lemma 5.4]{KRsf} to obtain
\begin{multline}\label{var short int d}
\Var(\EN_{\alpha_{2};h})    = \frac{H}{q^n} \sum_{m_1,m_2=0}^n
\alpha_{2}(T^{n-m_1}) \overline{\alpha_{2}(T^{n-m_2})} \\
\times \frac{1}{\Phi^{ev}(T^{n-h})} \sum_{\substack{\chi \bmod
T^{n-h}\\\chi\neq \chi_0\, {\rm even}}} \EM(m_1;\alpha_{2}\chi)
 \overline{\EM(m_2;\alpha_{2}\chi)}
\end{multline}
By the definition of $\alpha_{2}$ we have $\alpha_{2}(T^{n-m})=1$ when $m\neq n-1$ and $\alpha_{2}(T^{n-m})=0$ when $m= n-1$, hence \eqref{begin var intervals} follows.
\end{proof}

To compute the variance, we need to obtain an expression for $\EM(n;\alpha_{2}\chi)$. Consider the generating function
\begin{equation}
\sum_{n=0}^{\infty}\EM(n;\alpha_{2}\chi)u^{n}=\frac{L(u^{2},\chi^{2})L(u^{3},\chi^{3})}{L(u^{6},\chi^{6})}
\end{equation}
For an even primitive characters $\chi$ mod $T^{n-h}$, $L(u,\chi)$ has a trivial zero at $u=1$, hence we may write
\begin{equation}\label{l even}
L(u,\chi)=(1-u)\det(I-uq^{\frac{1}{2}}\Theta_{\chi})=(1-u)\sum_{i=0}^{n-h-2}u^{i}q^{i/2}\tr\Lambda_{i}(\Theta_{\chi})
\end{equation}
and
\begin{equation}\label{1divides l even}
\frac{1}{L(u,\chi)}=\frac{1}{(1-u)\det(I-uq^{\frac{1}{2}}\Theta_{\chi})}=\frac{1}{(1-u)}\sum_{i=0}^{\infty}u^{i}q^{i/2}\tr Sym^{i}(\Theta_{\chi})
\end{equation}
Therefore, the generating function of $\EM(n;\alpha_{2}\chi)$ (i.e. $\frac{L(u^{2},\chi^{2})L(u^{3},\chi^{3})}{L(u^{6},\chi^{6})}$) for $\chi$ mod $T^{n-h}$ such that $\chi^{2}, \chi^{3}, \chi^{6}$ are primitive and even can be written as follows
\begin{equation}
\begin{split}
&\frac{(1-u^{2})(1-u^{3})\det(I-u^{2}q^{\frac{1}{2}}\Theta_{\chi^{2}})\det(I-u^{3}q^{\frac{1}{2}}\Theta_{\chi^{3}})}{(1-u^{6})\det(I-u^{6}q^{\frac{1}{2}}\Theta_{\chi^{6}})}
=\\
&\frac{(1-u^{2})}{(1+u^{3})}
\sum_{j=0}^{n-h-2}\sum_{l=0}^{n-h-2}\sum_{k=0}^{\infty}u^{2j+3l+6k}q^{\frac{j+k+l}{2}}\Lambda_{j}(\chi^{2})\Lambda_{l}(\chi^{3})Sym^{k}(\chi^{6})
\end{split}
\end{equation}
By expanding and comparing coefficients we have
\begin{equation}
\EM(m;\alpha_{2}\chi)=\mathcal{S'}_{\chi}(m)-\mathcal{S'}_{\chi}(m-2)
\end{equation}
where for $1\neq m\geq0$
\begin{equation}
\mathcal{S'}_{\chi}(m):=\sum_{\substack{2j+3l+6k+3i=m\\0\leq j,l \leq n-h-2 \\0\leq k,i}}(-1)^{i}q^{\frac{j+k+l}{2}}\Lambda_{j}(\chi^{2})\Lambda_{l}(\chi^{3})Sym^{k}(\chi^{6})
\end{equation}
and
\begin{equation}
\mathcal{S'}_{\chi}(1),\mathcal{S'}_{\chi}(-1),\mathcal{S'}_{\chi}(-2):=0.
\end{equation}

Now back to the variance formula (see equation \eqref{begin var intervals}), we can split the sum into two parts: the sum over $\chi\neq\chi_{0} \mod T^{n-h},\chi\in\Gamma_{prim}^{ev}(T^{n-h})$ and the sum over even non-primitive characters mod $T^{n-h}$.
%Recall that the number of even characters modulo $T^{n-h}$ is
%$\Phi(T^{n-h})/(q-1) =q^{n-h-1}$ and $\Phi_{prim}^{ev}(T^{n-h})=q^{n-h-2}(q-1)$(see section \ref{dirichlet characters and Katz's equidistribution results} ).
We start by considering the first sum which will give the main term since most of the even characters are also primitive. With the second sum which will give an error term we deal later.
Note: by Lemma \ref{chi d prim} it is enough to split to these sums, and we can still use \eqref{1divides l even} and \eqref{l even} for $\chi^{2},\chi^{3},\chi^{6}.$

For $\chi$ even and primitive, consider the inner sum in the variance formula:

\begin{equation}
\begin{split}
  \sum_{\substack{ m_{1},m_{2}=0\\m_{1},m_{2}\neq n-1}}^{n}\EM(m_{1};\alpha_{2}\chi)\overline{\EM(m_{2};\alpha_{2}\chi)}&=\sum_{m_{1},m_{2}=0}^{n-2}\EM(m_{1};\alpha_{2}\chi)\overline{\EM(m_{2};\alpha_{2}\chi)}+|\EM(n;\alpha_{2}\chi)|^{2}
 \\&+ \overline{\EM(n;\alpha_{2}\chi)}\sum_{m=0}^{n-2}\EM(m;\alpha_{2}\chi)+\EM(n;\alpha_{2}\chi)\sum_{m=0}^{n-2}\overline{\EM(m;\alpha_{2}\chi)}
\end{split}
\end{equation}
The sum over $\EM(m;\alpha_{2}\chi)$ equals
\begin{equation}
\begin{split}
  \sum_{m=0}^{n-2}\EM(m;\alpha_{2}\chi)&=
  \sum_{m=0}^{n-2}(\mathcal{S'}_{\chi}(m)-\mathcal{S'}_{\chi}(m-2)) =
  (\mathcal{S'}_{\chi}(n-2)+\mathcal{S'}_{\chi}(n-3))
\end{split}
\end{equation}
Hence we get
\begin{equation}
\begin{split}
\sum_{\substack{ m_{1},m_{2}=0\\m_{1},m_{2}\neq n-1}}^{n}\EM(m_{1};\alpha_{2}\chi)\overline{\EM(m_{2};\alpha_{2}\chi)}=
 |\mathcal{S'}_{\chi}(n)|^{2}+|\mathcal{S'}_{\chi}(n-3)|^{2}
 &+\mathcal{S'}_{\chi}(n-3)\overline{\mathcal{S'}_{\chi}(n)}
 \\&+\overline{\mathcal{S'}_{\chi}(n-3)}\mathcal{S'}_{\chi}(n)
  \end{split}
\end{equation}
By using the equidistribution result \ref{k-e-second}, the average of $\mathcal{S'}_{\chi}(m_{1})\overline{\mathcal{S'}_{\chi}(m_{1})}$ over all primitive even characters mod $T^{n-h}$ in the limit $q\rightarrow\infty$ is
\begin{equation}
\begin{split}
 & \frac{1}{\Phi^{ev}(T^{n-h})} \sum_{\substack{\chi \bmod
T^{n-h}\\\chi\neq\chi_{0} \in \Gamma_{prim}^{ev}}}
\mathcal{S'}_{\Theta_{\chi}}(m_{1})\overline{\mathcal{S'}_{\Theta_{\chi}}(m_{2})}\sim\\&
\iiint_{PU(n-h-2)}
\sum_{\substack{2j+3l+6k+3i=m_{1}\\0\leq j,l \leq n-h-2\\ 0\leq i,k}}(-1)^{i}q^{\frac{j+k+l}{2}}\tr\Lambda_{j}(U_{1})\tr\Lambda_{l}(U_{2})\tr \sym^{k}(U_{3})\times
\\&\sum_{\substack{2j+3l+6k+3i=m_{2}\\0\leq j,l \leq n-h-2\\ 0\leq i,k}}(-1)^{i}q^{\frac{j+k+l}{2}}\overline{\tr\Lambda_{j}(U_{1})\tr\Lambda_{l}(U_{2})\tr \sym^{k}(U_{3})}\,dU_{1}\,dU_{2}\,dU_{3}
\end{split}
\end{equation}
By \eqref{ort exterior} and \eqref{ort symmetric}, we conclude that
\begin{equation}\label{diagonal}
\begin{split}
 & \frac{1}{\Phi^{ev}(T^{n-h})} \sum_{\substack{ m_{1},m_{2}=0\\m_{1},m_{2}\neq n-1}}^{n}\EM(m_{1};\alpha_{2}\chi)\overline{\EM(m_{2};\alpha_{2}\chi)}=\\&
  \sum_{\substack{2j+3l+6k+3i=n\\0\leq j,l \leq n-h-2\\ 0\leq i,k}}q^{j+k+l}+
  \sum_{\substack{2j+3l+6k+3i=n-3\\0\leq j,l \leq n-h-2\\ 0\leq i,k}}q^{j+k+l}
  -2\sum_{\substack{2j+3l+6k+3i=n-3\\0\leq j,l \leq n-h-2\\ 0\leq i,k}}q^{j+k+l}
  =
  \\& \sum_{\substack{2j+3l+6k=n\\0\leq j,l \leq n-h-2\\ 0\leq k}}q^{j+k+l}
\end{split}
\end{equation}
this combined with Lemma \ref{eval s(n)} (with $N=n-h-2$) gives the main term. Now, in order to complete the proof, it remains to bound the contribution of the even characters which are non-primitive to the variance.
Note that we do not have quadratic or cubic characters here since $\Phi^{ev}(T^{m})=q^{m-1}$ and we are considering the case of characteristic bigger than $13$ (see \ref{k-e-second}),
therefore there cannot be any even characters mod $T^{m}$ of order $2$ or $3$. For even characters which are non-primitive we still have the bound \eqref{bound_M} with $N=n-h-2$,
and since their proportion is $O(1/q)$ in the set of even characters, then we can bound their contribution as in the previous section,
therefore we skip the verification. Note: when bounding the contribution of non-primitive even characters we clearly don't use the equidistribution theorem as we do for the main term,
therefore we cannot restrict to diagonal terms as in \eqref{diagonal}. However, this does not change the bound since the off-diagonal terms do not contribute higher order terms then
the diagonal terms and their number does not depend on $q$.
%%%%%%%%%
%%%%%%%%%%
%%%%%%%%%%%%

\section*{Acknowledgements}
The author gratefully acknowledges support under EPSRC Programme Grant EP/K034383/1 LMF: L-Functions and Modular Forms. 
The author would like to thank Zeev Rudnick for suggesting this problem and to both Jon Keating and Zeev Rudnick for helpful discussions and remarks.

%%%%%%%%%%%%%%%%%%%%%%%%%%%%%%%%%%%%%%%%%%%%%%%%%%%%%%%%%%%%%%%%%%%%%%%%%%%%%%%%%%%%%%%%%%%%%%%%%%

{99}


\begin{thebibliography}{99}


\bibitem{Bateman}
Bateman, Paul T.; Grosswald, Emil.
{\em On a theorem of Erdös and Szekeres}.
Illinois J. Math. 2 (1958), no. 1, 88--98.


\bibitem{Cai}
Y. Cai,
{\em On the distribution of square-full integers}.
 Acta Math. Sinica (N.S.) 13
(1997), 269–280. A Chinese summary appears in Acta Math. Sinica 40 (1997),480.

\bibitem{Cao1}
 X.-D. Cao,
 {\em The distribution of square-full integers},
 Period. Math. Hungar. 28 (1994), 43–54.

\bibitem{Cao2}
 X. Cao,
 {\em On the distribution of square-full integers},
 Period. Math. Hungar. 34 (1997), 169–175.



\bibitem{Erdos}
P. Erd$\ooo$s and S. Szekeres , {\em $\ddot{U}$ber die Anzahl der Abelschen Gruppen gegebener
Ordnung und $\ddot{u}$ber ein verwandtes zahlentheoretisches Problem},
Acta Univ.Szeged 7 (1934–1935), 95–102.


\bibitem{Filaseta}
M.~ Filaseta and O.~ Trifonov.
{\em The distribution of fractional parts with applications to gap results in number theory}.
Proc. London Math. Soc. (3) 73 (1996), no. 2, 241–278.

\bibitem{HEATH-BROWN}
D.~ R. ~ Heath-Brown,  {\em Square-full numbers in short intervals}.
Square-full numbers in short intervals. Math. Proc. Cambridge Philos. Soc. 110 (1991), no. 1, 1–3.

\bibitem{Huxley}
M. N. Huxley and O. Trifonov (1996). {\em The square-full numbers in an interval}. Mathematical Proceedings of the Cambridge Philosophical Society, 119, pp 201-208 doi:10.1017/S0305004100074107


\bibitem{K2}
N.~Katz.
{\em Witt vectors and a question of Keating and Rudnick}.
Int. Math. Res. Not. IMRN 2013, no. 16, 3613–3638.


\bibitem{K1}
N.~Katz.
{\em On a question of Keating and Rudnick about primitive Dirichlet characters with squarefree conductor}.
Int. Math. Res. Not. IMRN 2013, no. 14, 3221–3249.


\bibitem{K4}
N.~Katz.
{\em Witt vectors and a question of Entin, Keating, and Rudnick}.
Int. Math. Res. Not. IMRN 2015, no. 14, 5959–5975.
doi: 10.1093/imrn/rnu120


\bibitem{K3}
N. ~Katz.
{\em On two questions of Entin, Keating, and Rudnick on primitive Dirichlet characters}.
Int. Math. Res. Not. IMRN 2015, no. 15, 6044–6069.
doi: 10.1093/imrn/rnu121


\bibitem{KR}
J.~P.~Keating and Z.~Rudnick. {\em The variance of the number of
prime polynomials in short intervals and in residue classes}.
Int. Math. Res. Notices, 2012; doi:
10.1093/imrn/rns220.

\bibitem{KRsf}
J.~P.~Keating and Z.~Rudnick. {\em Squarefree polynomials and
M\"obius values in short intervals and arithmetic progressions}.
Algebra and Number Theory, vol 10., pp. 375-420

\bibitem{Liu1}
H.-Q. Liu, {\em The distribution of square-full integers},
Ark. Mat. 32 (1994), 449–454.


\bibitem{Liu2}
H.-Q. Liu. {\em The number of squarefull numbers in an interval}.
Acta Arith. 64 (1993), no. 2, 129–149.


\bibitem{Liu and Zhang}
H.~Liu, and T.~Zhang,
{\em On the distribution of square-full numbers in arithmetic progressions}.
Arch. Math. (Basel) 101 (2013), no. 1, 53–64.


\bibitem{Munsch}
M.~Munsch,
{\em Character sums over squarefree and squarefull numbers}.
Arch. Math. (Basel) 102 (2014), no. 6, 555–563.

\bibitem{Suryanarayana}
 D. Suryanarayana and R. Sitamachandra,
 {\em The distribution of square-full integers},
  Ark. Mat. (1973), 195–201.

\bibitem{Wu1}
J. Wu, {\em On the distribution of square-full integers}, Arch. Math. (Basel) 77 (2001),
233–240.

\bibitem{Wu2}
J. Wu, {\em On the distribution of square-full and cube-full integers},
 Monatsh. Math. 126 (1998), 353–367.

\bibitem{Zhu-Yu}
W. Zhu and K. Yu, The distribution of square-full integers, Pure Appl. Math.
(Xi’an) 12 (1996), 113–122.

\end{thebibliography}
\end{document}